\documentclass[11pt,twoside]{amsart}
\usepackage{tabularx}
\usepackage{geometry}
\usepackage{amssymb,amsmath,amsthm,latexsym}
\usepackage{mathrsfs}
\usepackage{url}
\usepackage{amsfonts}
\usepackage{graphicx}
\usepackage{hyperref}

\usepackage{graphicx}
\usepackage{amssymb,amsthm,amsmath}
\usepackage{mathtools}

\newtheorem{theorem}{Theorem}[section]
\newtheorem{corollary}{Corollary}[section]
\newtheorem{lemma}{Lemma}[section]

\newtheorem{remark}{Remark}[section]
\usepackage{enumerate}

\usepackage{hyperref}
\usepackage{xcolor}

\hypersetup{colorlinks=true,linkcolor=blue,urlcolor=red}
\vspace{5cm}
  
\begin{document}

\title{A note on solitary numbers} 
\author[Sagar Mandal]{Sagar Mandal}
\address{Department of Mathematics and Statistics, Indian Institute of Technology Kanpur\\ Kalyanpur, Kanpur, Uttar Pradesh 208016, India}
\email{sagarmandal31415@gmail.com}

\maketitle
\let\thefootnote\relax
\let\thefootnote\relax
\maketitle

\begin{abstract}

Does $14$ have a friend? Until now, this has been an open question. In this note, we prove that a potential friend $F$ of $14$ is an odd, non-square positive integer. $7$ appears in the prime factorization of $F$ with an even exponent while at most two prime divisors of $F$ can have odd exponents in the prime factorization of $F$. If $p\mid F$ such that $p$ is congruent to $7$ modulo $8$, then $p^{2a}\mid\mid F$, for some positive integer $a$. Further, no prime divisor of $F$ has an exponent  congruent to $7$ modulo $8$ and no prime divisor can exceed $1.4\sqrt{F}$. The primes $3,5$ cannot appear simultaneously  in the prime factorization of $F$. If $(3,F)>1$ or $(5,F)>1$, then $\omega(F)\geq4$, otherwise $\omega(F)\geq8$.

\end{abstract}

\section{Introduction}
In number theory, the sum of divisors function $\sigma(n)$ plays a central role in studying the properties of integers. For a positive integer $n$, the abundancy index is defined as
$I(n) = \frac{\sigma(n)}{n}.$ More generally,
abundancy index can be considered as a measure of perfection of an integer, the abundancy index can be used to classify numbers as perfect, abundant, or deficient. A number is perfect if $I(n)=2$, abundant if $I(n)>2$, and deficient if $I(n)<2$. Two distinct positive integers $m$ and $n$ are called friends if they share the same abundancy index, that is, $I(m) = I(n).$
For example, all perfect numbers (OEIS A000396) are friends of each other, since they all have abundancy index $2$. If a number has no friend, it is called solitary. It is easy to prove \cite{1} that if a positive integer $n$ is co-prime to $\sigma(n)$, then $n$ is a solitary number (for example, see OEIS A014567). Anderson and Hickerson \cite{1} stated that the density of such solitary numbers is zero.  Although the concept of friendly numbers is simple, many interesting and difficult problems remain unsolved. It is not known whether $10$ has a friend, though many necessary conditions have been proposed \cite{4,5,7,8}. Moreover, among numbers less than $100$, those known to have friends are
$$6, \;12, \;24,\; 28,\; 30,\; 40,\; 42,\; 56,\; 60,\; 66,\; 78,\; 80,\; 84,\; 96.$$
For further results on the subject, see \cite{1,2,3,9}.

Recent works have focused on specific integers and their possible friends. For example, Ward \cite{8} investigated whether $10$ has a friend, while Terry \cite{6} studied the case of $15$. The status of several positive integers, including $14$, $15$, and $20$, is still unresolved. If any of the suspected solitary numbers up to $372$ is actually a friendly number, then its smallest friend must be strictly greater than $10^{30}$ (see OEIS A074902). In this paper, we investigate friends of $14$ and we give certain properties of a potential friend of $14$.

\section{Properties of the abundancy index}
Some  elementary properties of the abundancy index are given below, and the proofs of the following lemmas may be found in \cite{3,9}. 

\begin{lemma}\label{Lemma 2.1}
$I(n)$ is weakly multiplicative, that is, for any two co-prime positive integers $n$ and $m$ we have $I(nm)=I(n)I(m)$.
  \end{lemma}

\begin{lemma}\label{Lemma 2.2}
If $\gamma,n$ are two positive integers and $\gamma>1$. Then $I(\gamma n)>I(n)$.
\end{lemma}

\begin{lemma}\label{Lemma 2.3}
If $p_1, p_2,p_3,\dots,p_k$ are $k$ distinct prime numbers and $\gamma_1,\gamma_2,\gamma_3,\ldots,\gamma_k$ are positive integers, then
\begin{align*}
I\biggl (\prod_{i=1}^{k}p_i^{\gamma_i}\biggl)=\prod_{i=1}^{k}\biggl(\sum_{j=0}^{\gamma_i}p_i^{-j}\biggl)=\prod_{i=1}^{k}\frac{p_i^{\gamma_i+1}-1}{p_i^{\gamma_i}(p_i-1)}.
\end{align*}
\end{lemma}

\begin{lemma}\label{Lemma 2.4}
If $p_{1},\dots,p_{k}$ are distinct prime numbers, $q_{1},\ldots,q_{k}$ are distinct prime numbers such that  $p_{i}\leq q_{i}$ for all $1\leq i\leq k$. If $\gamma_1,\gamma_2,\dots,\gamma_k$ are positive integers, then
\begin{align*}
I \biggl(\prod_{i=1}^{k}p_i^{\gamma_i}\biggl)\geq I\biggl(\prod_{i=1}^{k}q_i^{\gamma_i}\biggl).
\end{align*}

\end{lemma}

\begin{lemma}\label{Lemma 2.5}
If $n=\displaystyle{\prod_{i=1}^{k}p_i^{\gamma_i}}$, then $\displaystyle{I(n)<\prod_{i=1}^{k}\frac{p_i}{p_i-1}}$.
\end{lemma}

Throughout this article, we use $p,p_1,\dots,p_{\omega(F)},p_k$ for denoting prime numbers. Further, we assume that the numbers $a_1,a_2,\dots,a_{\omega(F)},a_k$ are positive integers.

\section {Main results}

Note that, $I(14)=\frac{12}{7}$, therefore, a positive integer $F$ is said to be a friend of $14$ if $I(F)=I(14)=\frac{12}{7}$. The following results describe certain characteristics of a friend of $14$.
\begin{theorem}\label{thm 3.1}
Let $F$ be a friend of $14$, then $F$ is an odd positive non-square integer.   
\end{theorem}
\begin{proof}
    Let a positive integer $F$ be a friend of $14$. Then $F$ must be greater than $14$ as for any positive integer less than $14$ we have $I(F)\neq \frac{12}{7}$. Since $\frac{\sigma(F)}{F}=I(F)=\frac{12}{7}$ we have $7\sigma(F)=12F$, from which it follows that $7\mid F$ and $12\mid \sigma(F)$ as $(7,12)=1$. Therefore, we can write $F=7F'$ for some positive integer $F'>2$. 

Let us  assume that $F'$ is an even positive integer. Then we can rewrite $F$ as $F=14F''$ for some positive integer $F''>1$, but then $I(F)>I(14)$ by Lemma \ref{Lemma 2.2}. Therefore, $F'$ is an odd positive integer and thus $F$ is an odd positive integer.

    Now if $F=7^a$ for some positive integer $a>2$, then it cannot be a friend of $14$ as by Lemma~\ref{Lemma 2.5} we have $I(7^{a})<\frac{7}{6}<\frac{12}{7}$. Therefore, $F$ must be written as $\displaystyle{F=7^{a_{1}}\cdot \prod_{i=2}^{\omega(F)}p_{i}^{a_{i}}}$($p_1=7$). 

Let us suppose that all $a_i$ are even. Then
$$I(F)=\frac{\sigma(F)}{F}=\frac{12}{7},$$
using Lemma \ref{Lemma 2.1} we get
\begin{align*}
   I(7^{a_1})\cdot \prod_{i=2}^{\omega(F)}I(p_{i}^{a_i})=\frac{12}{7}.
   \end{align*}
   This implies
   \begin{align*}
 \sigma(7^{a_1})\cdot \prod_{i=2}^{\omega(F)}\sigma(p_i^{a_i})=12\cdot7^{{a_1}-1}\cdot\prod_{i=2}^{\omega(F)}p_{i}^{a_i},
\end{align*}
that is
\begin{align*}
 (1+7+ \cdots +7^{a_1})\cdot\prod_{i=2}^{\omega(F)}(1+p_i+\dots+p_i^{a_i})=12\cdot7^{{a_1}-1}\cdot\prod_{i=2}^{\omega(F)}p_{i}^{a_i}.
\end{align*}
Since $p_i>2$, for all $1\leq i\leq \omega(F)$ the right-hand side of the above expression is an even integer but the left-hand side is odd since $a_i$ are even, which immediately implies that $(1+p_i+\dots+p_i^{a_i})$ is odd, which is absurd. Hence, all $a_i$ cannot be even integers. Therefore, $F$ is a non-square positive integer. This proves that $F$ is an odd positive non-square integer.
\end{proof}

\begin{remark}\label{rem3.1}
If $F$ is a friend of $14$, then from the proof of Theorem \ref{thm 3.1}, we can note that $4\mid\mid\sigma(F)$, as $F$ is an odd positive integer.
\end{remark}

\medskip

Remark \ref{rem3.1} is very crucial as we will be using it enormously in the upcoming proofs.

\medskip

\begin{theorem}\label{thm 3.2}
Let $F$ be a friend of $14$. If $p\mid F$ such that $p\equiv 7 \pmod 8$, then $p$ appears in the prime factorization of $F$ to an even exponent.    
\end{theorem}
\begin{proof}
Suppose, for contradiction, that $p$ congruent to $7$ modulo $8$ is a prime divisor of $F$ with an odd exponent, that is, $p^{2a+1}\mid\mid F$, for some non-negative integer $a$. Then
$$\sigma(p^{2a+1})=1+p+p^2+\dots+p^{2a+1}\equiv 1-1+1 - \dots-1 \ (\bmod \ 8) =0.$$
This implies that $8\mid \sigma(p^{2a+1})$ and so $8\mid \sigma(F)$, but this is a contradiction by Remark~\ref{rem3.1}. Therefore, the exponent of $p$ in the prime factorization of $F$ must be an even positive integer.
\end{proof}

\smallskip

An immediate consequence of the preceding theorem, the following corollary describes that the exponent of the prime divisor $7$ of $F$ cannot be odd.

\smallskip

\begin{corollary}
If $F$ is a friend of $14$, then $7^{2a}\mid\mid F$, for some positive integer $a$.  
\end{corollary}
\begin{proof}
Since $F$ is a friend of $14$, $7$ is a prime divisor of $F$. As $7$ is a prime that satisfies \linebreak $7\equiv7~(\bmod~8)$, it follows from Theorem \ref{thm 3.2} that  $7^{2a}\mid\mid F$, for some positive integer $a$.
\end{proof}

\begin{theorem}
If $F$ is a friend of $14$ , then no prime divisor of $F$ has an exponent congruent to $7$ modulo $8$.
\end{theorem}

\begin{proof}
Let us assume that $p$ is a prime divisor of $F$ with an exponent congruent to $7$ modulo $8$, that is, $p^a\mid\mid F$ where $a\equiv7\ (\bmod \ 8)$. Then
\begin{align*}
    \sigma(p^a)=1+p+\dots+p^{a}\equiv \left\{
\begin{array}{ll}
      1+(\pm1)+\dots+(\pm1)^a \ (\bmod \ 8) \\[1mm]
     1+(\pm 3)+\dots+(\pm3)^a \ (\bmod \ 8)
\end{array}
\right.
\end{align*}
since $a\equiv 7 \pmod 8$ we have
$$1+(\pm1)+\dots+(\pm1)^a =\left\{\begin{array}{ll}
      a+1 \equiv 0 \ (\bmod \ 8)\\[1mm]
    0\equiv 0 \ (\bmod \ 8)
\end{array}\right. $$
and
$$1+(\pm 3)+\dots+(\pm3)^a =\left\{\begin{array}{ll}
      \displaystyle{\frac{3^{a+1}-1}{2} \equiv 0 \ (\bmod \ 8)} \\[3mm]
    \displaystyle{\frac{1-3^{a+1}}{4}\equiv 0 \ (\bmod \ 8)} 
\end{array}\right. . $$
This shows that $8\mid  \sigma(p^a)$, which implies $8\mid\sigma(F)$ but $4\mid\mid\sigma(F)$ from Remark \ref{rem3.1}. Hence, no prime divisor of $F$ has an exponent congruent to $7$ modulo $8$.
\end{proof}

\smallskip

We may ask how many distinct prime divisors of $F$ can have odd exponents in the prime factorization of $F$? The following theorem answers the question.

\smallskip

\begin{theorem}
If $F$ is a friend of $14$, then at most two distinct prime divisors of $F$ have odd exponents in the prime factorization of $F$.
\end{theorem}

\begin{proof}
Suppose, for contradiction, that $F$ has three distinct prime divisors $p_1,p_2,p_3$ with odd exponents $a_1,a_2,a_3$, respectively, in the prime factorization of $F$. Then
$$\sigma(p_i^{a_i})=1+p_i+\dots +p_i^{a_i}\equiv1+1+\dots+1 \ (\bmod \ 2) =0.$$
This implies that $2\mid \sigma(p_i^{a_i})$, for $i=1,2,3$. Thus, we get $8\mid \sigma(p_1^{a_1})\sigma(p_2^{a_2})\sigma(p_3^{a_3})$, that is, $8\mid \sigma(F)$ but this contradicts Remark \ref{rem3.1}. Therefore, we conclude that at most two distinct prime divisors of $F$ can have odd exponents in the prime factorization of $F$.
\end{proof}

\begin{theorem}
If $3$ is a divisor of a friend $F$ of $14$, then $3\mid\mid F$.
\end{theorem}

\begin{proof}
Let $F=3^{a}\cdot7^{2b}\cdot m$ be a friend of $14$, where $a,b,m$ are positive integers. If $a\geq3$, then using Lemma \ref{Lemma 2.2} we get
$$I(3^{a}\cdot7^{2b}\cdot m)\geq I(3^3\cdot 7^2)=\frac{760}{441}>I(14).$$
Therefore $a\leq2$. Let us suppose that $a=2$. Then
$$\frac{\sigma(3^2\cdot7^{2b}\cdot m)}{3^2\cdot7^{2b}\cdot m}=I(3^2\cdot7^{2b}\cdot m)=\frac{12}{7},$$
which is equivalent to
$$\sigma(3^2)\cdot \sigma(7^{2b})\cdot \sigma(m)=12\cdot3^2\cdot 7^{2b-1}\cdot m $$
since $\sigma(3^2)=13$, $13\mid 12\cdot3^2\cdot 7^{2b-1}\cdot m$, that is, $13\mid m$. Therefore, let $m=13m'$, for some positive integer $m'$. Then we have $F=3^2\cdot7^{2b}\cdot 13m'$. Using Lemma \ref{Lemma 2.2}, we get that
$$I(F)\geq I(3^2\cdot 7^2\cdot 13)=\frac{38}{21}>I(14).$$
Hence $a$ cannot be $2$. This completes the proof.
\end{proof}

\begin{lemma}\label{lemma 3.8}
If $F$ is a friend of $14$, then $3$ and $5$ cannot appear simultaneously in the prime factorization of $F$.
\end{lemma}

\begin{proof}
Let $F$ be a friend of $14$ and if possible, assume that $3,5$ appears simultaneously in the prime factorization of $F$. Then, using Lemma \ref{Lemma 2.2}, we have
$$I(F)\geq I(3\cdot5\cdot 7^2)=\frac{456}{245}>I(14).$$
Therefore, it follows that either $3$ or $5$ can appear in the prime factorization of $F$, but not \linebreak together.
\end{proof}

\smallskip

We now give the lower bounds for $\omega(F)$ according to the prime divisors of $F$.

\smallskip

\begin{theorem}
 If $F$ is a friend of $14$, then $\omega(F)\geq 4$ whenever $3\mid F$ or $5\mid F$. Further, if $(3,F)=(5,F)=1$, then  $\omega(F)\geq 8$.   
\end{theorem}
\begin{proof}
Let $F$ be a friend of $14$. If $3\mid F$, then $5\nmid F$ by Lemma \ref{lemma 3.8}, therefore all prime divisors of $F$ are greater than $5$. Let us suppose that $F$ has exactly three distinct prime divisors, that is, $F=3\cdot 7^{2a}\cdot p^{b}$, where $p>7$ is a prime and $a,b$ are positive integers. Then, using Lemma \ref{Lemma 2.1}, Lemma \ref{Lemma 2.4} and Lemma \ref{Lemma 2.5}, we get 
$$I(F)\leq I(3\cdot 7^{2a}\cdot 11^{b})=I(3)\cdot I(7^{2a}\cdot 11^{b})<\frac{4}{3}\cdot \frac{7}{6}\cdot\frac{11}{10}=\frac{77}{45}<I(14).$$
Therefore, $F$ cannot have exactly three distinct prime divisors, hence $\omega(F)\geq 4$.

If $5\mid F$, then $3\nmid F$ by Lemma \ref{lemma 3.8}, therefore all prime divisors of $F$ are greater than $3$. Let us suppose that $F$ has exactly three distinct prime divisors, that is, $F=5^{a}\cdot 7^{2b}\cdot p^{c}$, where $p>7$ is a prime and $a,b,c\in\mathbb{Z^+}$. Then, using Lemma \ref{Lemma 2.4} and Lemma \ref{Lemma 2.5}, we get 
$$I(F)\leq I(5^a\cdot 7^{2b}\cdot 11^{c})<\frac{5}{4}\cdot \frac{7}{6}\cdot\frac{11}{10}=\frac{77}{48}<I(14).$$
Therefore, $F$ cannot have exactly three distinct prime divisors, hence $\omega(F)\geq 4$.

Let  $(3,F)=(5,F)=1$. Then every prime divisors of $F$ are strictly greater than $5$. Let us suppose that $F$ has at most seven distinct prime divisors, that is, $F=\displaystyle{7^{2a}\cdot\prod_{i=1}^{k}p_i^{a_i}}$, where $p_{i+1}>p_i>7$ and $k\leq 6$. Then, by Lemma\ref{Lemma 2.2}, Lemma \ref{Lemma 2.4}, and Lemma \ref{Lemma 2.5}, we get
\begin{align*}
I(F)&\leq I(7^{2a}\cdot \prod_{i=5}^{k+4}q_i^{a_{i-4}})~~~~~~~(q_i \text{ is the $i$-th prime number})\\
&\leq  I(7^{2a}\cdot 11^{a_1}\cdot 13^{a_2}\cdot 17^{a_3}\cdot 19^{a_4}\cdot 23^{a_5}\cdot 29^{a_6})\\[2mm]
&<\frac{7\cdot11\cdot13\cdot17\cdot19\cdot23\cdot29}{6\cdot10\cdot12\cdot16\cdot18\cdot22\cdot28}\\[2mm]
& =\frac{2800733}{1658880} \\
& <I(14).    
\end{align*}

Therefore, $F$ cannot have at most seven distinct prime divisors, hence $\omega(F)\geq 8$. This completes the proof.
\end{proof}

\begin{theorem}
No prime divisor of a friend $F$ of $14$ can exceed $1.4\sqrt{F}$. \end{theorem}
\begin{proof}
Let $p$ be a prime divisor of $F$. Then we can write $F=p^{a}\cdot7^{2b}\cdot m$ where $a,b,m$ are positive integers such that $(7p,m)=1$. Since
$$I(F)=\frac{\sigma(F)}{F}=\frac{\sigma(p^a)\cdot\sigma(7^{2b})\cdot\sigma(m)}{p^a\cdot7^{2b}\cdot m}=\frac{12}{7},$$
we have
$$\sigma(F)=\sigma(p^a)\cdot\sigma(7^{2b})\cdot\sigma(m)=12\cdot p^a\cdot7^{2b-1}\cdot m=\frac{12F}{7}.$$
Note that, $p^a\mid \sigma(F)$ and $\sigma(p^a)\mid \sigma(F)$, since $(p^a,\sigma(p^a))=1$ we have $p^a\cdot\sigma(p^a)\mid \sigma(F)$. Therefore,
$$p^2\leq p^a\cdot\sigma(p^a)\leq \sigma(F)=\frac{12F}{7},$$
that is
$$p\leq \sqrt{\frac{12F}{7}}<1.4\sqrt{F}.$$
For the prime divisor $7$ of $F$, the proof proceeds in the same way.
\end{proof}

\makeatletter
\renewcommand{\@biblabel}[1]{[#1]\hfill}
\makeatother

\end{document}